\documentclass[11pt]{amsart}

\usepackage[T1]{fontenc}
\usepackage[utf8]{inputenc}
\usepackage{lmodern}
\usepackage{microtype}
\usepackage{amsmath,amssymb,mathtools,mathrsfs}
\usepackage{booktabs,array,graphicx}
\usepackage[a4paper,margin=30mm]{geometry}
\usepackage[colorlinks=true,linkcolor=blue,citecolor=blue,urlcolor=blue]{hyperref}

\newtheorem{theorem}{Theorem}[section]
\newtheorem{proposition}[theorem]{Proposition}
\newtheorem{lemma}[theorem]{Lemma}
\newtheorem{corollary}[theorem]{Corollary}

\theoremstyle{remark}
\newtheorem{remark}[theorem]{Remark}

\newcommand{\Des}{\operatorname{Des}}
\newcommand{\des}{\operatorname{des}}
\newcommand{\fixB}{\operatorname{fix}_B}
\newcommand{\tauB}{\operatorname{cyc}_2}
\newcommand{\oc}{\operatorname{oc}}
\newcommand{\NN}{\mathbb N}
\newcommand{\Par}{\operatorname{Par}}

\title[Fixed-point Worpitzky identity in type $B$]{A Fixed-Point Worpitzky Identity and a Positive Binomial Transform for Type $B$ Involutions}
\author[J. Zeng]{Jiang Zeng}

\address{College of Mathematics and Physics,
Wenzhou University, Wenzhou 325035, China; Universit\'e Claude Bernard
Lyon~1, CNRS UMR~5208, Institut Camille Jordan, France.}
\email{zeng@math.univ-lyon1.fr}
\date{September 2026}

\begin{document}
\begin{abstract}
Let $\mathcal I_n^B$ be the involutions of the hyperoctahedral group $\mathfrak B_n$, and let $\des^B$ denote the descent number with respect to the natural Coxeter order. We derive the fixed-point-refined Worpitzky identity
\[
\sum_{n\ge0}\frac{\mathcal F_n(p,t)\,z^n}{(1-t)^{n+1}}
=\sum_{m\ge0}
\frac{(1+pz)^m\,t^m}{(1-pz)^{m+1}(1-z^2)^{m(m+1)}},
\quad
\mathcal F_n(p,t)=\sum_{\pi\in\mathcal I_n^B}p^{\fixB(\pi)}t^{\des^B(\pi)}.
\]
Extracting the stratum with $j$ two-cycles and $f$ fixed positions yields a one-parameter deformation of the fixed-point-free Worpitzky series of Wan, Gao, Li and Yang. After the change of variables $x=t/(1+t)^2$, this deformation becomes a positive binomial transform. More precisely, if
\[
P_j(x)=\sum_sD_{2j,s}x^s
\]
is the fixed-point-free $\gamma$-polynomial, then the transform coefficients $A_{j,r}(x)$ are determined by
\[
\sum_{r\ge0}A_{j,r}(x)W^r
=\sum_{s=0}^{j}D_{2j,s}x^s
(1+4xW)^{2j-2s}(1+2W+4xW^2)^s,
\]
and the $\gamma$-polynomial of the $(j,f)$-stratum is
\[
\Phi_{j,f}(x)=\sum_{r=0}^{f}\binom fr A_{j,r}(x).
\]
This manifestly positive transform is the main structural result of the paper. As consequences, every fixed cycle-type stratum is $\gamma$-positive and $\mathcal F_n(p,t)$ is coefficientwise $\gamma$-positive in the fixed-point variable $p$. The cases $f=0$ and $p=1$ recover, respectively, the fixed-point-free theorem of Wan--Gao--Li--Yang and the all-involution theorem of Cao--Liu. We also give explicit formulas for the first binomial layers and for the strata with one and two two-cycles.
\end{abstract}
\maketitle

\section{Introduction}

The Eulerian distribution on involutions has a long history of symmetry, unimodality, and $\gamma$-positivity questions. In type \(A\), Guo and Zeng proved the unimodality of the descent polynomial on involutions of the symmetric group and conjectured the stronger $\gamma$-positivity property~\cite{GuoZeng2006}. Motivated by Brenti's stronger conjecture that this descent polynomial is log-concave, Dukes subsequently studied several refined permutation statistics on involutions, with particular reference to questions and conjectures concerning unimodality and log-concavity~\cite{Dukes2007}. Brenti's conjecture was later disproved by Barnabei, Bonetti, and Silimbani~\cite{BarnabeiBonettiSilimbani2009}, whereas the Guo--Zeng $\gamma$-positivity conjecture was eventually proved by Wang~\cite{Wang2019}.

The corresponding questions for type \(B\) involutions were initiated by Moustakas, who proved that the descent polynomial on involutions of the hyperoctahedral group is unimodal and obtained a Worpitzky-type generating function by means of signed quasisymmetric functions and the type \(B\) Robinson--Schensted correspondence~\cite{Moustakas2019}. Cao and Liu subsequently settled the corresponding $\gamma$-positivity problem by proving that this descent polynomial is $\gamma$-positive~\cite{CaoLiu2021}.

There are two natural total orders on signed letters. When $\mathfrak B_n$ is viewed as a Coxeter group of type $B$, the relevant order is
\[
\cdots<-2<-1<0<1<2<\cdots,
\]
which we call the \emph{natural order}. When $\mathfrak B_n$ is viewed as a $2$-colored permutation group, the $r$-order is often more convenient. Adin, Brenti and Roichman observed that the natural-order and $r$-order descent numbers are equidistributed over the whole hyperoctahedral group $\mathfrak B_n$ \cite[p.~218]{AdinBrentiRoichman2001}. Gao, Li, Wan and Yang subsequently showed that this equidistribution persists on the involution subset $\mathcal I_n^B$, but fails on the fixed-point-free subclass \cite{GaoLiWanYang2023}. Cao and Liu proved $\gamma$-positivity for the fixed-point-free distribution in the $r$-order \cite{CaoLiu2023}. For the natural order, Wan, Gao, Li and Yang established symmetry, unimodality and $\gamma$-positivity \cite{WanGaoLiYang2024}.

The purpose of the present paper is not only to retain the number of fixed positions in these positivity results, but to identify the structural mechanism that connects the fixed-point-free boundary to all fixed-point strata. For a signed involution $\pi$, we call $i$ a fixed position whenever $\pi(i)=\pm i$, following \cite{CaoLiu2023,WanGaoLiYang2024}, and write
\[
\fixB(\pi)=\#\{i\in[n]:|\pi(i)|=i\}.
\]
Define
\begin{equation}\label{eq:main-bivar}
\mathcal F_n(p,t)
=\sum_{\pi\in\mathcal I_n^B}p^{\fixB(\pi)}t^{\des^B(\pi)}.
\end{equation}
Our first result is a fixed-point refinement of the Worpitzky kernel itself.

\begin{theorem}[Fixed-point-refined Worpitzky identity]\label{thm:refined-worp}
For $\mathcal F_n(p,t)$ defined by \eqref{eq:main-bivar},
\begin{equation}\label{eq:refined-worp}
\sum_{n\ge0}\frac{\mathcal F_n(p,t)}{(1-t)^{n+1}}z^n
=\sum_{m\ge0}t^m
\frac{(1+pz)^m}{(1-pz)^{m+1}(1-z^2)^{m(m+1)}}.
\end{equation}
\end{theorem}

The two endpoints of this kernel are already meaningful: $p=0$ gives the fixed-point-free natural-order series of Wan--Gao--Li--Yang, whereas $p=1$ gives Moustakas' Worpitzky kernel. More importantly, coefficient extraction from \eqref{eq:refined-worp} separates the fixed cycle-type strata and leads to a positive transform between their $\gamma$-polynomials.

Every signed involution has an underlying ordinary involution on $[n]$. If it has $j$ two-cycles and $f$ fixed positions, then $n=2j+f$. Let
\begin{equation}\label{eq:Fjf-def}
F_{j,f}(t)
=\sum_{\substack{\pi\in\mathcal I_{2j+f}^B\\
\tauB(\pi)=j,\ \fixB(\pi)=f}}
t^{\des^B(\pi)}.
\end{equation}
Thus
\begin{equation}\label{eq:decomp-p}
\mathcal F_n(p,t)
=\sum_{j=0}^{\lfloor n/2\rfloor}p^{n-2j}F_{j,n-2j}(t).
\end{equation}
Put
\begin{equation}\label{eq:intro-gamma-change}
x=\frac{t}{(1+t)^2},
\qquad
F_{j,f}(t)=(1+t)^{2j+f}\Phi_{j,f}(x).
\end{equation}
For $f=0$, $F_{j,0}(t)$ is the fixed-point-free natural-order polynomial $J_{2j}^B(t)$ of Wan--Gao--Li--Yang. Write
\begin{equation}\label{eq:WG-LY-gamma}
J_{2j}^B(t)
=(1+t)^{2j}P_j\!\left(\frac{t}{(1+t)^2}\right),
\qquad
P_j(x)=\sum_{s=0}^{j}D_{2j,s}x^s,
\end{equation}
where $D_{2j,s}\ge0$ by \cite[Theorem~6]{WanGaoLiYang2024}.

The following positive binomial transform is the main structural result of the paper.

\begin{theorem}[Positive binomial transform]\label{thm:A-formula}
For every $j\ge0$, define polynomials $A_{j,r}(x)$ by
\begin{equation}\label{eq:A-GF}
\sum_{r\ge0}A_{j,r}(x)W^r
=(1+4xW)^{2j}
P_j\!\left(\frac{x(1+2W+4xW^2)}{(1+4xW)^2}\right).
\end{equation}
Then
\begin{equation}\label{eq:A-GF-positive}
\sum_{r\ge0}A_{j,r}(x)W^r
=\sum_{s=0}^{j}D_{2j,s}x^s
(1+4xW)^{2j-2s}(1+2W+4xW^2)^s,
\end{equation}
and therefore
\begin{equation}\label{eq:A-positive}
A_{j,r}(x)\in\NN[x].
\end{equation}
Moreover,
\begin{equation}\label{eq:Phi-binomial}
\Phi_{j,f}(x)=\sum_{r=0}^{f}\binom fr A_{j,r}(x).
\end{equation}
\end{theorem}

Writing
\begin{equation}\label{eq:A-jkr}
A_{j,r}(x)=\sum_k A_{j,k,r}x^k,
\end{equation}
we obtain the positivity statements as consequences of this transform rather than as isolated phenomena.

\begin{corollary}[Binomial-$\gamma$ and coefficientwise fixed-point $\gamma$-positivity]\label{cor:main-positivity}
For all $j,f\ge0$,
\begin{equation}\label{eq:Fjf-gamma}
F_{j,f}(t)
=\sum_{k\ge0}\gamma_{j,f,k}t^k(1+t)^{2j+f-2k},
\end{equation}
where
\begin{equation}\label{eq:binomial-gamma}
\gamma_{j,f,k}
=\sum_{r=0}^{f}A_{j,k,r}\binom fr,
\qquad A_{j,k,r}\in\NN.
\end{equation}
Consequently, for every $n\ge0$ there exist polynomials $C_{n,k}(p)\in\NN[p]$ such that
\begin{equation}\label{eq:bivar-gamma}
\mathcal F_n(p,t)
=\sum_{k=0}^{\lfloor n/2\rfloor}C_{n,k}(p)t^k(1+t)^{n-2k}.
\end{equation}
\end{corollary}

The natural order is essential for the coefficientwise statement. Although natural-order and $r$-order descents are equidistributed on all signed involutions after setting $p=1$, their fixed-point refinements are not coefficientwise equidistributed; in fact, the $r$-order analogue of Corollary~\ref{cor:main-positivity} already fails for $n=2$. We return to this point in Section~\ref{sec:conseq}.

Thus the paper gives more than a new $\gamma$-positive polynomial: it identifies a positive operator taking the fixed-point-free boundary polynomial $P_j(x)$ to every fixed-point stratum. The binomial coefficient $\binom fr$ separates the dependence on the number of fixed positions from the positive coefficients encoded by \eqref{eq:A-GF-positive}. This transform also makes the boundary cases transparent: $f=0$ recovers Wan--Gao--Li--Yang's theorem, while summing the cycle-type strata at fixed $n$, equivalently setting $p=1$, recovers the all-involution $\gamma$-positivity theorem of Cao--Liu under the natural-order/$r$-order equidistribution.

The paper is organized as follows. Section~\ref{sec:prelim} fixes notation and recalls the fixed-point-free input. Section~\ref{sec:worp} proves Theorem~\ref{thm:refined-worp} from the type $B$ Robinson--Schensted correspondence and refined Littlewood identities. Section~\ref{sec:strata} extracts the fixed cycle-type strata from the refined Worpitzky kernel. Section~\ref{sec:gamma} derives the $\gamma$-polynomial generating function and proves the positive binomial transform of Theorem~\ref{thm:A-formula}. Section~\ref{sec:conseq} develops consequences of the transform, including explicit formulas for its first layers and for the complete $j=1,2$ strata. We conclude with questions about combinatorial, $q$-, and wreath-product extensions of the transform.

\section{Preliminaries}\label{sec:prelim}

\subsection{Signed involutions and natural descents}
For $n\ge1$, the hyperoctahedral group $\mathfrak B_n$ consists of signed permutations $\pi=\pi_1\cdots\pi_n$ such that $|\pi_1|,\ldots,|\pi_n|$ is a permutation of $[n]$. We put $\pi_0=0$ and use the natural order on $\mathbb Z$. The type $B$ descent set and descent number are
\[
\Des^B(\pi)=\{i\in\{0,1,\ldots,n-1\}:\pi_i>\pi_{i+1}\},
\qquad
\des^B(\pi)=|\Des^B(\pi)|.
\]
Let $\mathcal I_n^B$ be the involutions of $\mathfrak B_n$.

For $\pi\in\mathcal I_n^B$, the unsigned permutation $i\mapsto |\pi(i)|$ is an ordinary involution. A fixed position is an $i$ with $|\pi(i)|=i$. The remaining positions occur in two-cycles, and therefore
\begin{equation}\label{eq:n-2j-f}
2\tauB(\pi)+\fixB(\pi)=n.
\end{equation}

The following elementary symmetry will also be useful.
\begin{lemma}\label{lem:symmetry}
For fixed $j,f$, the polynomial $F_{j,f}(t)$ is symmetric with center $(2j+f)/2$.
\end{lemma}
\begin{proof}
The map $\pi\mapsto-\pi$ is a bijection of the relevant cycle-type stratum and preserves $\fixB$. Since every comparison in
\[
0,\pi_1,\ldots,\pi_{2j+f}
\]
is reversed, one has
\[
\des^B(-\pi)=2j+f-\des^B(\pi).
\]
\end{proof}

\subsection{The fixed-point-free boundary}
Let
\[
\mathcal J_{2j}^B
=\{\pi\in\mathcal I_{2j}^B:\fixB(\pi)=0\},
\qquad
J_{2j}^B(t)=\sum_{\pi\in\mathcal J_{2j}^B}t^{\des^B(\pi)}.
\]
Wan, Gao, Li and Yang proved the following Worpitzky identity \cite[Theorem~3]{WanGaoLiYang2024}:
\begin{equation}\label{eq:WG-Worpitzky}
\frac{J_{2j}^B(t)}{(1-t)^{2j+1}}
=\sum_{m\ge0}
\binom{m^2+m+j-1}{j}t^m.
\end{equation}
They further proved that $J_{2j}^B(t)$ is $\gamma$-positive \cite[Theorem~6]{WanGaoLiYang2024}. We write its $\gamma$-polynomial as in \eqref{eq:WG-LY-gamma}:
\begin{equation}\label{eq:Pj}
P_j(x)=\sum_{s=0}^{j}D_{2j,s}x^s,
\qquad D_{2j,s}\in\NN.
\end{equation}
Their first values are
\begin{align*}
P_0(x)&=1,\\
P_1(x)&=2x,\\
P_2(x)&=3x,\\
P_3(x)&=4x+12x^2+8x^3,\\
P_4(x)&=5x+51x^2+132x^3+16x^4.
\end{align*}

\section{The fixed-point-refined Worpitzky identity}\label{sec:worp}

The starting point is a refinement of Moustakas' type $B$ kernel \cite[Theorem~1.1]{Moustakas2019}.

\subsection{Odd columns and fixed positions}

We first explain how the fixed-position statistic is recorded by the
type $B$ Robinson--Schensted correspondence.

For a signed permutation $\pi\in\mathfrak B_n$, write
\[
 \fixB^+(\pi)=\#\{i\in[n]:\pi(i)=i\},
 \qquad
 \fixB^-(\pi)=\#\{i\in[n]:\pi(i)=-i\}.
\]
Thus
\begin{equation}\label{eq:fix-plus-minus}
 \fixB(\pi)=\fixB^+(\pi)+\fixB^-(\pi).
\end{equation}

Recall that the type $B$ Robinson--Schensted correspondence associates
with a signed permutation $\pi$ a pair of standard Young bitableaux
\[
 \pi\longleftrightarrow (P(\pi),Q(\pi)),
 \qquad
 P(\pi)=(P^+,P^-),\quad Q(\pi)=(Q^+,Q^-),
\]
of the same bipartition shape
\[
 \operatorname{sh}(\pi)=(\lambda,\mu),
 \qquad |\lambda|+|\mu|=n.
\]
One way to describe this correspondence is to split the signed
permutation matrix of $\pi$ into its positive and negative parts.
Namely, let $M^+(\pi)$ have a $1$ in position $(i,j)$ when
$\pi(i)=j$, and let $M^-(\pi)$ have a $1$ in position $(i,j)$ when
$\pi(i)=-j$. Ordinary RSK, applied to these two partial permutation
matrices, produces respectively the two components of the bitableaux.

Suppose now that $\pi\in\mathcal I_n^B$. Since $\pi=\pi^{-1}$,
both $M^+(\pi)$ and $M^-(\pi)$ are symmetric. Indeed,
\[
 \pi(i)=j\quad\Longrightarrow\quad \pi(j)=i,
\]
whereas
\[
 \pi(i)=-j\quad\Longrightarrow\quad \pi(j)=-i.
\]
Consequently the insertion and recording tableaux agree in each
component, and a signed involution corresponds to a single standard
Young bitableau
\[
 T=(T^+,T^-)
\]
of shape $(\lambda,\mu)$.

The diagonal $1$'s of $M^+(\pi)$ are precisely the positive fixed
points $\pi(i)=i$, whereas the diagonal $1$'s of $M^-(\pi)$ are
precisely the negative fixed points $\pi(i)=-i$. For a symmetric
partial permutation matrix, the classical trace theorem for RSK says
that the number of diagonal $1$'s is equal to the number of columns
of odd length of its Young diagram; see Knuth \cite{Knuth1970}.
Hence, if $\oc(\nu)$ denotes the number of odd columns of a partition
$\nu$, we have
\begin{equation}\label{eq:positive-negative-fix}
 \fixB^+(\pi)=\oc(\lambda),
 \qquad
 \fixB^-(\pi)=\oc(\mu).
\end{equation}
Combining \eqref{eq:fix-plus-minus} and
\eqref{eq:positive-negative-fix} gives the fundamental relation
\begin{equation}\label{eq:fix-oddcols}
\fixB(\pi)=\oc(\lambda)+\oc(\mu).
\end{equation}

We next recall the symmetric-function identities that keep track of
these odd columns. For an alphabet $X=(x_1,x_2,\ldots)$, the
trace-refined Littlewood identity is
\begin{equation}\label{eq:Littlewood1}
 \sum_{\lambda\in\Par}
 u^{\oc(\lambda)}s_\lambda(X)
 =
 \frac{1}
 {\displaystyle
  \prod_i(1-ux_i)\prod_{i<j}(1-x_ix_j)}.
\end{equation}
This is the refinement of Littlewood's identity corresponding, under
RSK, to the number of diagonal entries; see Goulden
\cite{Goulden1992}. Its dual form is
\begin{equation}\label{eq:Littlewood2}
 \sum_{\lambda\in\Par}
 u^{\oc(\lambda)}s_{\lambda'}(X)
 =
 \frac{\displaystyle\prod_i(1+ux_i)}
 {\displaystyle\prod_{i\le j}(1-x_ix_j)}.
\end{equation}

It remains to explain why the first identity will be used for the
positive component and the second for the negative component. The
type $B$ Robinson--Schensted correspondence is descent preserving:
the natural-order descent set of a signed permutation is carried to
the corresponding descent set of its recording bitableau; see
\cite{AdinAthanasiadisElizaldeRoichman2017,Moustakas2019}.
For the natural Coxeter order
\[
 \cdots<-2<-1<0<1<2<\cdots,
\]
the symmetric-function encoding of the negative component involves
the usual involution
\[
 \omega(s_\mu)=s_{\mu'}.
\]
Thus, at Worpitzky level $m$, a bitableau of bipartition shape
$(\lambda,\mu)$ contributes
\[
 s_\lambda(1^{m+1})\,s_{\mu'}(1^m).
\]
The two different alphabet sizes $m+1$ and $m$ reflect the type $B$
descent convention, including the possible descent at position $0$.

Therefore, after introducing a variable $p$ for fixed positions and
a variable $z$ for size, the contribution of all signed involutions
at level $m$ is
\begin{equation}\label{eq:schur-product}
 \left(
  \sum_{\lambda}
  p^{\oc(\lambda)}
  s_\lambda(1^{m+1})z^{|\lambda|}
 \right)
 \left(
  \sum_{\mu}
  p^{\oc(\mu)}
  s_{\mu'}(1^m)z^{|\mu|}
 \right).
\end{equation}
Formula \eqref{eq:fix-oddcols} is precisely what makes this product a
fixed-point refinement of the kernel occurring in Moustakas'
Worpitzky identity.

We can now prove the refined Worpitzky identity stated in the introduction.

\begin{proof}[Proof of Theorem~\ref{thm:refined-worp}]
By the type $B$ Robinson--Schensted specialization described above,
the contribution at Worpitzky level $m$ is \eqref{eq:schur-product}.
We now evaluate its two factors using the refined Littlewood identities.

Specializing \eqref{eq:Littlewood1} to $m+1$ variables all equal to $z$ gives
\[
\sum_\lambda p^{\oc(\lambda)}s_\lambda(1^{m+1})z^{|\lambda|}
=\frac{1}{(1-pz)^{m+1}(1-z^2)^{\binom{m+1}{2}}}.
\]
Similarly, \eqref{eq:Littlewood2} with $m$ variables equal to $z$ gives
\[
\sum_\mu p^{\oc(\mu)}s_{\mu'}(1^{m})z^{|\mu|}
=\frac{(1+pz)^m}{(1-z^2)^{\binom{m+1}{2}}}.
\]
Multiplying these two expressions yields
\[
\frac{(1+pz)^m}
{(1-pz)^{m+1}(1-z^2)^{m(m+1)}}.
\]
Summing over the Worpitzky level with weight $t^m$ proves \eqref{eq:refined-worp}.
\end{proof}

\begin{remark}\label{rem:endpoints-kernel}
At $p=0$, Theorem~\ref{thm:refined-worp} reduces to
\[
\sum_{j\ge0}\frac{J_{2j}^B(t)}{(1-t)^{2j+1}}z^{2j}
=\sum_{m\ge0}\frac{t^m}{(1-z^2)^{m(m+1)}},
\]
which is the generating function of Wan--Gao--Li--Yang. At $p=1$,
\[
\frac{(1+z)^m}{(1-z)^{m+1}(1-z^2)^{m(m+1)}}
=\frac1{(1-z)^{2m+1}(1-z^2)^{m^2}},
\]
recovering Moustakas' kernel.
\end{remark}

\section{From the Worpitzky kernel to fixed cycle-type strata}\label{sec:strata}

Let
\begin{equation}\label{eq:cfm}
c_f(m)=[u^f]\frac{(1+u)^m}{(1-u)^{m+1}}.
\end{equation}
Equivalently,
\begin{equation}\label{eq:cfm-explicit}
c_f(m)=\sum_{a=0}^{\min(f,m)}
\binom ma\binom{m+f-a}{f-a}.
\end{equation}

\begin{proposition}[Stratum Worpitzky identity]\label{prop:stratum-worp}
For all $j,f\ge0$,
\begin{equation}\label{eq:stratum-worp}
\frac{F_{j,f}(t)}{(1-t)^{2j+f+1}}
=\sum_{m\ge0}c_f(m)
\binom{m^2+m+j-1}{j}t^m.
\end{equation}
\end{proposition}

\begin{proof}
Extract $p^fz^{2j+f}$ from \eqref{eq:refined-worp}. The factor involving $p$ contributes
\[
[p^fz^f]\frac{(1+pz)^m}{(1-pz)^{m+1}}=c_f(m),
\]
while
\[
[z^{2j}](1-z^2)^{-m(m+1)}
=\binom{m^2+m+j-1}{j}.
\]
\end{proof}

Proposition~\ref{prop:stratum-worp} isolates the effect of fixed positions particularly cleanly: for fixed $j$, the fixed-point-free Worpitzky kernel
\[
\binom{m^2+m+j-1}{j}
\]
is unchanged, and all dependence on the number $f$ of fixed positions is carried by the single multiplier $c_f(m)$. This separation is what ultimately becomes the binomial transform in Section~\ref{sec:gamma}.

The case $f=0$ has $c_0(m)=1$ and is exactly \eqref{eq:WG-Worpitzky}. The first two nontrivial multipliers are
\begin{equation}\label{eq:c1c2}
c_1(m)=2m+1,
\qquad
c_2(m)=2m(m+1)+1.
\end{equation}

It is convenient to package all $f$ simultaneously.

\begin{proposition}\label{prop:fixedpoint-GF}
Let
\[
\mathscr F_j(y,t)=\sum_{f\ge0}F_{j,f}(t)y^f
\]
and set
\begin{equation}\label{eq:T-def}
T=\frac{t\{1+y(1-t)\}}{1-y(1-t)}.
\end{equation}
Then
\begin{equation}\label{eq:Fj-y-t}
\mathscr F_j(y,t)
=\frac{\{1-y(1-t)\}^{2j}}
{\{1-y(1+t)\}^{2j+1}}
J_{2j}^B(T).
\end{equation}
\end{proposition}

\begin{proof}
Multiply \eqref{eq:stratum-worp} by $u^f$ and sum over $f$. From \eqref{eq:cfm},
\[
\sum_{f\ge0}\frac{F_{j,f}(t)}{(1-t)^{2j+f+1}}u^f
=\frac1{1-u}\sum_{m\ge0}
\binom{m^2+m+j-1}{j}
\left(t\frac{1+u}{1-u}\right)^m.
\]
Use \eqref{eq:WG-Worpitzky} with $T=t(1+u)/(1-u)$ and then put $u=y(1-t)$. A direct simplification, using
\[
1-T=(1-t)\frac{1-y(1+t)}{1-y(1-t)},
\]
gives \eqref{eq:Fj-y-t}.
\end{proof}

\section{The positive binomial transform}\label{sec:gamma}

We now convert the Worpitzky deformation of Section~\ref{sec:strata} into an operator on $\gamma$-polynomials. The point of the change of variables below is that the fixed-point parameter, which appears rationally in \eqref{eq:Fj-y-t}, becomes an ordinary binomial transform after the substitution \eqref{eq:W-sub}.

Put
\begin{equation}\label{eq:x-def}
x=\frac{t}{(1+t)^2},
\qquad
a=1-4x=\frac{(1-t)^2}{(1+t)^2}.
\end{equation}
By Lemma~\ref{lem:symmetry}, there is a unique polynomial $\Phi_{j,f}(x)$ such that
\begin{equation}\label{eq:Phi-def}
F_{j,f}(t)=(1+t)^{2j+f}\Phi_{j,f}(x).
\end{equation}
Thus the coefficients of $\Phi_{j,f}$ are precisely the $\gamma$-coefficients in \eqref{eq:Fjf-gamma}.

\begin{proposition}[Generating function for the $\gamma$-polynomials]\label{prop:Phi-GF}
For every $j\ge0$,
\begin{equation}\label{eq:Phi-GF}
\sum_{f\ge0}\Phi_{j,f}(x)s^f
=\frac{(1-as)^{2j}}{(1-s)^{2j+1}}
P_j\!\left(
\frac{x(1-as^2)}{(1-as)^2}
\right).
\end{equation}
\end{proposition}

\begin{proof}
In \eqref{eq:Fj-y-t}, put $s=y(1+t)$ and use \eqref{eq:WG-LY-gamma}. We have
\[
1-y(1+t)=1-s
\]
and
\[
1+t-y(1-t)^2=(1+t)(1-as).
\]
Moreover, for $T$ from \eqref{eq:T-def},
\[
\frac{T}{(1+T)^2}
=\frac{x(1-as^2)}{(1-as)^2}.
\]
After dividing by $(1+t)^{2j}$, equation \eqref{eq:Phi-GF} follows.
\end{proof}

The ordinary/binomial generating-function substitution
\begin{equation}\label{eq:W-sub}
W=\frac{s}{1-s},
\qquad
s=\frac{W}{1+W}
\end{equation}
will make positivity transparent.

The next step is to prove the positive transform announced in the introduction.

\begin{proof}[Proof of Theorem~\ref{thm:A-formula}]
For any sequence $\Phi_f$ one has
\begin{equation}\label{eq:binomial-OGF}
\Phi_f=\sum_r\binom fr A_r
\quad\Longleftrightarrow\quad
(1-s)\sum_{f\ge0}\Phi_fs^f
=\sum_{r\ge0}A_r\left(\frac{s}{1-s}\right)^r.
\end{equation}
Apply this to \eqref{eq:Phi-GF} and set $W=s/(1-s)$. Since
\[
\frac{1-as}{1-s}=1+4xW
\]
and
\[
\frac{1-as^2}{(1-as)^2}
=\frac{1+2W+4xW^2}{(1+4xW)^2},
\]
we obtain \eqref{eq:A-GF}. Substituting \eqref{eq:Pj} yields \eqref{eq:A-GF-positive}. Since every $D_{2j,s}$ is nonnegative by Wan--Gao--Li--Yang and every factor on the right of \eqref{eq:A-GF-positive} has nonnegative coefficients, \eqref{eq:A-positive} follows. Finally \eqref{eq:Phi-binomial} is the coefficient form of \eqref{eq:binomial-OGF}.
\end{proof}

Theorem~\ref{thm:A-formula} gives a fully explicit manifestly positive coefficient formula.

\begin{corollary}[Explicit positive formula]\label{cor:A-explicit}
For all $j,k,r\ge0$,
\begin{equation}\label{eq:A-explicit}
A_{j,k,r}
=\sum_{s=0}^{j}D_{2j,s}
\!\sum_{\substack{a,b,c\ge0\\
a+c=k-s\\
a+b+2c=r\\
b+c\le s}}
4^{a+c}2^b
\binom{2j-2s}{a}
\binom{s}{b,c,s-b-c}.
\end{equation}
In particular $A_{j,k,r}\in\NN$.
\end{corollary}

\begin{proof}
Expand the two factors in \eqref{eq:A-GF-positive}. Choose $a$ copies of $4xW$ from the first factor, and from the second factor choose $b$ copies of $2W$ and $c$ copies of $4xW^2$. The conditions in \eqref{eq:A-explicit} are exactly the resulting $x$- and $W$-degree constraints.
\end{proof}

\begin{proof}[Proof of Corollary~\ref{cor:main-positivity}]
By \eqref{eq:Phi-binomial} and \eqref{eq:A-positive}, the coefficients of $\Phi_{j,f}(x)$ are nonnegative. Writing \eqref{eq:A-jkr} gives
\[
\gamma_{j,f,k}=\sum_{r=0}^{f}\binom fr A_{j,k,r}\in\NN,
\]
which proves \eqref{eq:Fjf-gamma} and \eqref{eq:binomial-gamma}. Combining this with \eqref{eq:decomp-p} gives
\[
C_{n,k}(p)=\sum_{j=0}^{\lfloor n/2\rfloor}\gamma_{j,n-2j,k}p^{n-2j}\in\NN[p],
\]
and hence \eqref{eq:bivar-gamma}.
\end{proof}

\section{Consequences and examples}\label{sec:conseq}

\subsection{The first two binomial layers}
The coefficient of $W$ in \eqref{eq:A-GF-positive} gives a simple positive formula.

\begin{corollary}\label{cor:A1}
For all $j,k$,
\begin{equation}\label{eq:A1}
A_{j,k,1}
=2kD_{2j,k}
+8(j-k+1)D_{2j,k-1}.
\end{equation}
\end{corollary}

The coefficient of $W^2$ yields the next layer.

\begin{corollary}\label{cor:A2}
For all $j,k$,
\begin{equation}\label{eq:A2}
\begin{aligned}
A_{j,k,2}
={}&2k(k-1)D_{2j,k}\\
&+4(k-1)(4j-4k+5)D_{2j,k-1}\\
&+16(j-k+2)(2j-2k+3)D_{2j,k-2}.
\end{aligned}
\end{equation}
\end{corollary}

Both formulas are manifestly nonnegative in the relevant range. Thus the $r=0$ layer is precisely Wan--Gao--Li--Yang's Theorem~6, while \eqref{eq:A1} and \eqref{eq:A2} give the first two new layers of the refinement.

\subsection{The first $\gamma$-column}
The first fixed-point-free coefficient is $D_{2j,1}=j+1$; this follows, for example, from the recurrence in \cite[Theorem~2]{WanGaoLiYang2024}. Formula \eqref{eq:A-GF-positive} then shows that
\[
A_{j,1,0}=j+1,
\qquad
A_{j,1,1}=2(j+1),
\qquad
A_{j,1,r}=0\quad(r\ge2).
\]
Hence:

\begin{corollary}\label{cor:first-column}
For $j\ge1$ and $f\ge0$,
\begin{equation}\label{eq:first-column}
\gamma_{j,f,1}=(j+1)(2f+1).
\end{equation}
Equivalently, if $n=2j+f$,
\begin{equation}\label{eq:first-column-n}
\gamma_{j,n-2j,1}
=(j+1)(2n-4j+1).
\end{equation}
\end{corollary}

Since the lowest $t$-degree in a $\gamma$-expansion is triangular, \eqref{eq:first-column} also enumerates the signed involutions in the stratum with exactly one natural descent.

\subsection{One two-cycle}
For $j=1$, $P_1(x)=2x$, and \eqref{eq:A-GF-positive} becomes
\[
\sum_{r\ge0}A_{1,r}(x)W^r
=2x(1+2W+4xW^2).
\]
Therefore
\begin{equation}\label{eq:j1-Phi}
\Phi_{1,f}(x)
=2(2f+1)x+8\binom f2x^2.
\end{equation}
We obtain the closed form
\begin{corollary}\label{cor:j1}
For every $f\ge0$,
\begin{equation}\label{eq:j1}
F_{1,f}(t)
=2(2f+1)t(1+t)^f
+8\binom f2t^2(1+t)^{f-2}.
\end{equation}
\end{corollary}
If $n=f+2$, this is
\[
F_{1,n-2}(t)
=(4n-6)t(1+t)^{n-2}
+4(n-2)(n-3)t^2(1+t)^{n-4}.
\]

\subsection{Two two-cycles}
For $j=2$, $P_2(x)=3x$, and Theorem~\ref{thm:A-formula} gives
\begin{align}
\Phi_{2,f}(x)
={}&(3+6f)x
+\left(24f+60\binom f2\right)x^2\notag\\
&+\left(48\binom f2+192\binom f3\right)x^3
+192\binom f4x^4.
\label{eq:j2-Phi}
\end{align}
Hence:

\begin{corollary}\label{cor:j2}
For every $f\ge0$,
\begin{align}
F_{2,f}(t)
={}&(6f+3)t(1+t)^{f+2}\notag\\
&+\left(24f+60\binom f2\right)t^2(1+t)^f\notag\\
&+\left(48\binom f2+192\binom f3\right)t^3(1+t)^{f-2}\notag\\
&+192\binom f4t^4(1+t)^{f-4}.
\label{eq:j2}
\end{align}
\end{corollary}
This gives a complete closed-form proof of $\gamma$-positivity for the two-two-cycle stratum.

\subsection{A small table}
For illustration, the first coefficient polynomials $C_{n,k}(p)$ in \eqref{eq:bivar-gamma} are
\[
\begin{array}{c|llll}
\toprule
n&C_{n,0}&C_{n,1}&C_{n,2}&C_{n,3}\\
\midrule
2&p^2&2&&\\
3&p^3&6p&&\\
4&p^4&10p^2+3&8p^2&\\
5&p^5&14p^3+9p&24p^3+24p&\\
6&p^6&18p^4+15p^2+4&48p^4+108p^2+12&48p^2+8\\
\bottomrule
\end{array}
\]
Corollary~\ref{cor:main-positivity} explains all coefficientwise nonnegativity in this table.

\subsection{Boundary specializations and the role of the natural order}
At $p=0$, only the fixed-point-free stratum survives. Thus Corollary~\ref{cor:main-positivity} contains the natural-order fixed-point-free $\gamma$-positivity theorem of Wan--Gao--Li--Yang as a boundary case.

At $p=1$, one obtains
\[
\mathcal F_n(1,t)
=\sum_{\pi\in\mathcal I_n^B}t^{\des^B(\pi)}.
\]
The equidistribution of the natural-order and $r$-order descent numbers over $\mathfrak B_n$ was already observed by Adin--Brenti--Roichman \cite[p.~218]{AdinBrentiRoichman2001}. Gao--Li--Wan--Yang proved that the same equidistribution persists after restriction to the involution subset $\mathcal I_n^B$ \cite{GaoLiWanYang2023}. Hence the specialization $p=1$ agrees with the polynomial whose $\gamma$-positivity was proved by Cao--Liu \cite{CaoLiu2021}.

It is worth stressing that the analogous coefficientwise fixed-point refinement in the $r$-order cannot satisfy the same statement. Already for $n=2$ one obtains
\[
\sum_{\pi\in\mathcal I_2^B}p^{\fixB(\pi)}t^{\operatorname{des}_r(\pi)}
=p^2+(3p^2+1)t+t^2,
\]
which is not symmetric coefficientwise in $p$. Thus the natural order is essential for Corollary~\ref{cor:main-positivity}.

\section{Further questions}

The positive binomial transform in Theorem~\ref{thm:A-formula} is the structural core of the paper: it carries the fixed-point-free $\gamma$-polynomial $P_j(x)$ to every fixed-point stratum. Several questions about this transform remain natural.

\begin{enumerate}
\item \emph{Combinatorial interpretation of $A_{j,k,r}$.}
Formula \eqref{eq:A-explicit} is positive but algebraic. The binomial coefficient $\binom fr$ suggests that $r$ counts a distinguished set of ``essential'' fixed positions, while the remaining $f-r$ fixed positions are inert. It would be desirable to construct such a decomposition directly on signed involutions or signed permutation grids.

\item \emph{A direct action proof.}
Is there a Foata--Strehl-type or valley-hopping action on each fixed cycle-type stratum whose orbit polynomials are $t^k(1+t)^{2j+f-2k}$ and whose primitive objects are enumerated by \eqref{eq:A-explicit}?

\item \emph{$q$-refinements.}
The fixed-point variable is compatible with the ordinary $\gamma$-basis, while the most naive flag-major $q$-$\gamma$ refinement of signed involutions develops negative coefficients. It would be interesting to determine whether the positive transform in Theorem~\ref{thm:A-formula} admits a meaningful $q$-analogue at the level of normalized $q$-$g$ coefficients or two-alphabet symmetric functions.

\item \emph{Other wreath products.}
Moustakas' colored-quasisymmetric framework and the work on $k$-colored involutions suggest seeking a fixed-point-refined Worpitzky kernel for $\mathbb Z_r\wr\mathfrak S_n$, with separate variables marking fixed points of each color. Does its cycle-type extraction again factor through a positive binomial transform of the corresponding fixed-point-free $\gamma$-polynomial?
\end{enumerate}

\end{document}